\documentclass[a4paper,12pt]{article}
\usepackage{amsmath,amsthm,amssymb}
\usepackage{color,epsfig,graphics}

\numberwithin{equation}{section}
\newtheorem{theorem}{Theorem}
\newtheorem{proposition}{Proposition}
\newtheorem{lemma}{Lemma}

\theoremstyle{definition}
\newtheorem{remark}{Remark}
\newtheorem{definition}{Definition}

\newcommand{\C}{\mathbb{C}}
\newcommand{\R}{\mathbb{R}}
\newcommand{\Z}{\mathbb{Z}}
\newcommand{\sgn}{\mathop{\mathrm{sgn}}\nolimits}
\newcommand{\dual}[2]{\langle #1, #2\rangle}

\title{Bifurcation from semi-trivial standing waves and ground states 
for a system of nonlinear Schr\"odinger equations}

\author{Mathieu Colin\thanks{Institut de Math\'ematiques de Bordeaux, 
Universit\'e de Bordeaux and INRIA Bordeaux-Sud Ouest, EPI MC2, 
351 cours de la lib\'eration, 33405 Talence Cedex, France 
({\tt mcolin@math.u-bordeaux1.fr}).}
\hspace{1mm} and \hspace{1mm} 
Masahito Ohta\thanks{Institut de Math\'ematiques de Bordeaux, 
Universit\'e Bordeaux 1. 
Permanent address: Department of Mathematics, Faculty of Science, 
Saitama University, Saitama 338-8570, Japan 
({\tt mohta@mail.saitama-u.ac.jp}).}}

\begin{document}

\date{}
\maketitle

\begin{abstract}
We consider a system of nonlinear Schr\"odinger equations 
related to the Raman amplification in a plasma. 
We study the orbital stability and instability of standing waves 
bifurcating from the semi-trivial standing wave of the system. 
The stability and instability of the semi-trivial 
standing wave at the bifurcation point are also studied. 
Moreover, we determine the set of the ground states completely. 
\end{abstract}

\section{Introduction}

\subsection{Motivation}

In this paper, we consider the following system of nonlinear Schr\"odinger equations 
\begin{equation}\label{eq:co}
\left\{\begin{array}{l}
i\partial_tu_1=-\Delta u_1-\kappa |u_1|u_1-\gamma \overline{u_1}u_2 \\
i\partial_tu_2=-2\Delta u_2-2|u_2|u_2-\gamma u_1^2
\end{array}\right.
\end{equation}
for $(t,x)\in \R\times \R^N$, 
where $u_1$ and $u_2$ are complex-valued functions of $(t,x)$, 
$\kappa\in \R$ and $\gamma>0$ are constants and $N\le 3$. 
System \eqref{eq:co} is a reduced system studied in \cite{CCO1,CCO2} 
and related to the Raman amplification in a plasma. 
Roughly speaking, the Raman amplification is an instability phenomenon 
taking place when an incident laser field propagates into a plasma. 
We refer to \cite{CC1,CC2} for a precise description of the phenomenon.
A similar system to \eqref{eq:co} also appears as an optics model 
with quadratic nonlinearity (see \cite{yew}). 

In \cite{CCO1,CCO2}, the authors studied the following three-component system
\begin{equation}\label{eq:cco}
\left\{\begin{array}{l}
i\partial_tv_1=-\Delta v_1-|v_1|^{p-1}v_1-\gamma v_3\overline{v_2} \\
i\partial_tv_2=-\Delta v_2-|v_2|^{p-1}v_2-\gamma v_3\overline{v_1} \\
i\partial_tv_3=-\Delta v_3-|v_3|^{p-1}v_3-\gamma v_1v_2,
\end{array}\right.
\end{equation}
where $1<p<1+4/N$ and $N\le 3$. Let $\omega>0$ and 
let $\varphi_{\omega}\in H^1(\R^N)$ be a unique positive radial solution of 
\begin{equation}\label{scaP}
-\Delta \varphi+\omega \varphi-|\varphi|^{p-1}\varphi=0, \quad x\in \R^N. 
\end{equation}
Then, $(0,0,e^{i\omega t}\varphi_{\omega})$ solves \eqref{eq:cco}. 
We note that $e^{i\omega t}\varphi_{\omega}$ is a standing wave solution 
of the single nonlinear Schr\"odinger equation
\begin{equation}\label{nls}
i\partial_tu=-\Delta u-|u|^{p-1}u, \quad (t,x)\in \R\times \R^N, 
\end{equation}
and that $e^{i\omega t}\varphi_{\omega}$ is orbitally stable 
for \eqref{nls} if $1<p<1+4/N$, 
and it is unstable if $1+4/N\le p<1+4/(N-2)$ 
(see \cite{BC,CL} and also \cite[Chapter 8]{caz}). 
In \cite{CCO1,CCO2}, the authors proved the following result 
on the semi-trivial standing wave solution 
$(0,0,e^{i\omega t}\varphi_{\omega})$ of \eqref{eq:cco}. 

\vspace{2mm} \noindent
{\bf Theorem 0.} (\cite{CCO1,CCO2}) \hspace{1mm}\textit{
Let $N\le 3$, $1<p<1+4/N$, $\omega>0$, 
and let $\varphi_{\omega}$ be the positive radial solution of \eqref{scaP}. 
Then, there exists a positive constant $\gamma^*$ such that 
the semi-trivial standing wave solution $(0,0,e^{i\omega t}\varphi_{\omega})$ 
of \eqref{eq:cco} is stable if $0<\gamma<\gamma^*$, 
and it is unstable if $\gamma>\gamma^*$.
}\vspace{2mm} 

By the local bifurcation theorem by Crandall and Rabinowitz \cite{CR}, 
it is easy to see that $\gamma=\gamma^*$ is a bifurcation point. 
We are interested in the structure of the bifurcation 
from the semi-trivial standing wave of \eqref{eq:cco} and its stability property. 
However, this problem is difficult to study in the general case $1<p<1+4/N$, 
so we consider the special case $p=2$. 
Moreover, since $v_1$ and $v_2$ play the same role in the proof of Theorem 0, 
we consider a reduced system \eqref{eq:co} assuming $v_1=v_2$ in \eqref{eq:cco}. 
We also introduce a parameter $\kappa$ in the first equation of \eqref{eq:co}, 
which makes the structure of standing wave solutions richer as we will see below. 
We remark that the positive constant $\gamma^*$ in Theorem 0 is given by 
\begin{equation}\label{gamma*}
\gamma^*=\inf\left\{\frac{\|\nabla v\|_{L^2}^2+\omega \|v\|_{L^2}^2}
{\int_{\R^N}\varphi_{\omega}(x)|v(x)|^2\,dx}:
v\in H^1(\R^N)\setminus\{0\}\right\}.
\end{equation}
For the case $p=2$, since $\varphi_{\omega}$ is the positive radial solution of 
\begin{equation}\label{scalar}
-\Delta \varphi+\omega \varphi-|\varphi|\varphi=0, \quad x\in \R^N, 
\end{equation}
we see that the infimum in \eqref{gamma*} is attained at 
$v=\varphi_{\omega}$ and $\gamma^*=1$. 
In the same way as the proof of Theorem 0, we can prove the following. 

\begin{theorem}\label{thm1}
Let $N\le 3$, $\kappa\in \R$, $\gamma>0$, $\omega>0$, 
and let $\varphi_{\omega}$ be the positive radial solution of \eqref{scalar}. 
Then, the semi-trivial standing wave solution $(0,e^{2i\omega t}\varphi_{\omega})$ 
of \eqref{eq:co} is stable if $0<\gamma<1$, and it is unstable if $\gamma>1$. 
\end{theorem}

We remark that the stability property of the semi-trivial standing wave 
of \eqref{eq:co} is independent of $\kappa$ for the case $\gamma\ne 1$. 
On the other hand, we will see that the sign of $\kappa$ plays an important role 
for the case $\gamma=1$ (see Theorems \ref{thm4} and \ref{thm5} below). 

\subsection{Notation and Definitions}

Before we state our main results, we prepare some notation and definitions. 
For a complex number $z$, we denote by $\Re z$ and $\Im z$ its real and imaginary parts. 
Thoughout this paper, we assume that $N\le 3$. 
We regard $L^2(\R^N,\C)$ as a real Hilbert space with the inner product 
$$(u,v)_{L^2}=\Re \int_{\R^N}u(x)\overline{v(x)}\,dx,$$
and we define the inner products of real Hilbert spaces 
$H=L^2(\R^N,\C)^2$ and $X=H^1(\R^N,\C)^2$ by 
$$(\vec u,\vec v)_{H}=(u_1,v_1)_{L^2}+(u_2,v_2)_{L^2}, \quad 
(\vec u,\vec v)_{X}=(\vec u,\vec v)_{H}+(\nabla \vec u,\nabla \vec v)_{H}.$$
Here and hereafter, we use the vectorial notation $\vec u=(u_1,u_2)$, 
and it is considered to be a column vector. 

The energy $E$ and the charge $Q$ are defined by 
\begin{align*}
&E(\vec u)
=\frac{1}{2}\|\nabla \vec u\|_{H}^2
-\frac{\kappa}{3}\|u_1\|_{L^3}^3 -\frac{1}{3}\|u_2\|_{L^3}^3
-\frac{\gamma}{2} \Re \int_{\R^N}u_1^2\overline{u_2}\,dx, \\
&Q(\vec u)=\frac{1}{2}\|\vec u\|_{H}^2, \quad \vec u\in X. 
\end{align*}
For $\theta \in \R$, 
we define $G(\theta)$ and $J$ by 
$$G(\theta)\vec u=(e^{i\theta}u_1,e^{2i\theta}u_2), \quad 
J\vec u=(iu_1,2iu_2), \quad \vec u\in X,$$ 
and 
$$\dual{G(\theta)\vec f}{\vec u}=\dual{\vec f}{G(-\theta)\vec u}, \quad  
\dual{J\vec f}{\vec u}=-\dual{\vec f}{J\vec u}$$ 
for $\vec f\in X^*$ and $\vec u\in X$, where $X^*$ is the dual space of $X$. 
For $y\in \R^N$, we define 
$$\tau_y \vec u(x)=\vec u(x-y), \quad \vec u\in X,~ x\in \R^N.$$
Note that \eqref{eq:co} is written as 
$$\partial_t \vec u(t)=-JE'(\vec u(t)),$$ 
and that $E(G(\theta)\tau_y\vec u)=E(\vec u)$ 
for all $\theta\in \R$, $y\in \R^N$ and $\vec u\in X$. 

By the standard theory (see, e.g., \cite[Chapter 4]{caz}), 
we see that the Cauchy problem for \eqref{eq:co} 
is globally well-posed in $X$, 
and the energy and the charge are conserved. 
For $\omega>0$, we define the action $S_{\omega}$ by 
$$S_{\omega}(\vec v)=E(\vec v)+\omega Q(\vec v), \quad \vec v\in X.$$
Note that the Euler-Lagrange equation $S_{\omega}'(\vec \phi)=0$ is written as 
\begin{equation}\label{sp}
\left\{\begin{array}{l}
-\Delta \phi_1+\omega \phi_1=\kappa |\phi_1|\phi_1+\gamma \overline{\phi_1}\phi_2 \\
-\Delta \phi_2+\omega \phi_2=|\phi_2|\phi_2+({\gamma}/{2}) \phi_1^2
\end{array}\right.
\end{equation}
and that if $\vec \phi\in X$ satisfies $S_{\omega}'(\vec \phi)=0$, 
then $G(\omega t)\vec \phi$ is a solution of \eqref{eq:co}. 

\begin{definition}
We say that a standing wave solution $G(\omega t)\vec \phi$ of \eqref{eq:co}
is {\it stable} if for all $\varepsilon>0$ there exists $\delta>0$ 
with the following property. 
If $u_0\in X$ satisfies $\|\vec u_0-\vec \phi\|_X<\delta$, 
then the solution $\vec u(t)$ of \eqref{eq:co} with $\vec u(0)=\vec u_0$ 
exists for all $t\ge 0$, and satisfies 
$$\inf_{\theta\in \R,y\in \R^N}
\|\vec u(t)-G(\theta)\tau_y\vec \phi\|_X<\varepsilon$$ 
for all $t\ge 0$. 
Otherwise, $G(\omega t)\vec \phi$ is called {\it unstable}. 
\end{definition}

In this article, we are also interested in the classification of ground states of \eqref{sp}. 
A ground state of \eqref{sp} is a nontrivial solution which minimizes the action $S_\omega$ 
among all the nontrivial solutions of \eqref{sp}. 
The set $\mathcal{G}_{\omega}$ of the ground states for \eqref{sp} is then defined as follows:
\begin{align*}
&\mathcal{A}_{\omega}=\{\vec v\in X:S_{\omega}'(\vec v)=0,~ \vec v\ne 0\}, \\
&d(\omega)=\inf\{S_{\omega}(\vec v):\vec v\in \mathcal{A}_{\omega}\}, \\
&\mathcal{G}_{\omega}=\{\vec u\in \mathcal{A}_{\omega}:S_{\omega}(\vec u)=d(\omega)\}.
\end{align*}

\subsection{Main Results}

We first look for solutions of \eqref{sp} of the form 
$\vec \phi=(\alpha \varphi_{\omega},\beta \varphi_{\omega})$ 
with $(\alpha,\beta)\in ]0,\infty[^2$, 
where $\varphi_{\omega}$ is the positive radial solution of \eqref{scalar}. 
It is clear that if $(\alpha,\beta)\in ]0,\infty[^2$ satisfies 
\begin{equation}\label{eq:ab}
\kappa \alpha+\gamma \beta=1, \quad \gamma \alpha^2+2\beta^2=2\beta,
\end{equation}
then $(\alpha \varphi_{\omega},\beta \varphi_{\omega})$ is a solution of \eqref{sp}. 
For $\kappa\in \R$ and $\gamma>0$, we define 
$$\mathcal{S}_{\kappa,\gamma}
=\{(x,y)\in ]0,\infty[^2:
\kappa x+\gamma y=1, ~ \gamma x^2+2y^2=2y\}.$$
Note that $\gamma x^2+2y^2=2y$ is an ellipse with vertices $(x,y)=(0,0)$, 
$(0,1)$, $(\pm 1/\sqrt{2\gamma},1/2)$, 
and that $\mathcal{S}_{\kappa,\gamma}\subset \{(x,y):0<y<1\}$. 

To determine the structure of the set $\mathcal{S}_{\kappa,\gamma}$, 
which is one of the crucial points of our analysis, 
for $\kappa^2\geq 2\gamma(1-\gamma)$ we define 
\begin{align*}
&\alpha_{\pm}=\frac{(2-\gamma)\kappa\pm \gamma \sqrt{\kappa^2+2\gamma (\gamma-1)}}
{2\kappa^2+\gamma^3}, \\
&\beta_{\pm}=\frac{\kappa^2+\gamma^2 \pm \kappa \sqrt{\kappa^2+2\gamma (\gamma-1)}}
{2\kappa^2+\gamma^3}, \\
&\alpha_{0}=\frac{(2-\gamma)\kappa}{2\kappa^2+\gamma^3}, \quad 
\beta_{0}=\frac{\kappa^2+\gamma^2}{2\kappa^2+\gamma^3}. 
\end{align*}
We also divide the parameter domain 
$\mathcal{D}=\{(\kappa,\gamma):\kappa\in \R,~ \gamma>0\}$ into the following sets 
(see Figure \ref{figure:J}). 
\begin{align*}
&\mathcal{J}_1=\{(\kappa,\gamma):\kappa\le 0,~ \gamma>1\} \cup 
\{(\kappa,\gamma):\kappa>0,~ \gamma\ge 1\}, \\
&\mathcal{J}_2=\{(\kappa,\gamma):
0<\gamma<1,~ \kappa>\sqrt{2\gamma (1-\gamma)}\}, \\
&\mathcal{J}_3=\{(\kappa,\gamma):
0<\gamma<1,~ \kappa=\sqrt{2\gamma (1-\gamma)}\}, \\
&\mathcal{J}_0=\{(\kappa,\gamma):\kappa\in \R,~ \gamma>0\}
\setminus (\mathcal{J}_1\cup \mathcal{J}_2\cup \mathcal{J}_3). 
\end{align*}

\begin{figure}[htbp]
\begin{center}
\scalebox{0.4}{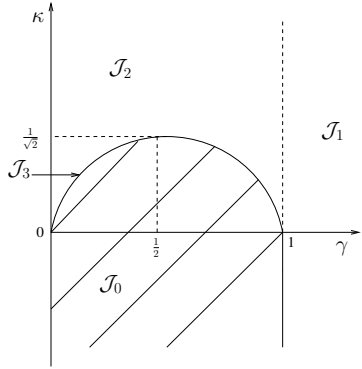 }
\caption{The sets $\mathcal{J}_0$, $\mathcal{J}_1$, 
$\mathcal{J}_2$ and $\mathcal{J}_3$}
\label{figure:J}
\end{center}
\end{figure}

Notice that the sets $\mathcal{J}_0$, $\mathcal{J}_1$, 
$\mathcal{J}_2$ and $\mathcal{J}_3$ are mutually disjoint, 
and $\mathcal{D}=\mathcal{J}_0\cup \mathcal{J}_1\cup \mathcal{J}_2\cup \mathcal{J}_3$. 
Note also that for $0<\kappa\le 1/\sqrt{2}$, the equation $2\gamma (1-\gamma)=\kappa^2$ 
has solutions $\gamma=\gamma_{\pm}:=(1\pm \sqrt{1-2\kappa^2})/2$. 
It is then possible to determine the set $\mathcal{S}_{\kappa,\gamma}$ 
in terms of $\alpha_{\pm}$, $\beta_{\pm}$, $\alpha_0$ and $\beta_0$.
Indeed, by elementary computations, we obtain the following. 

\begin{proposition}\label{prop1}
$({\rm 0})$ \hspace{1mm}
If $(\kappa,\gamma)\in \mathcal{J}_0$, then 
$\mathcal{S}_{\kappa,\gamma}$ is empty. 
\par \noindent $({\rm 1})$ \hspace{1mm}
If $(\kappa,\gamma)\in \mathcal{J}_1$, then 
$\mathcal{S}_{\kappa,\gamma}=\{(\alpha_{+},\beta_{-})\}$. 
\par \noindent $({\rm 2})$ \hspace{1mm}
If $(\kappa,\gamma)\in \mathcal{J}_2$, then 
$\mathcal{S}_{\kappa,\gamma}=\{(\alpha_{+},\beta_{-}),~ (\alpha_{-},\beta_{+})\}$. 
\par \noindent $({\rm 3})$ \hspace{1mm}
If $(\kappa,\gamma)\in \mathcal{J}_3$, then 
$\mathcal{S}_{\kappa,\gamma}=\{(\alpha_{0},\beta_{0})\}$. 
\end{proposition}

\begin{remark}
\par \noindent $({\rm 1})$ \hspace{1mm}
When $\kappa\le 0$, $(\alpha_{+},\beta_{-})\to (0,1)$ as $\gamma\to 1+0$. 
That is, the branch 
$\{(\alpha_{+}\varphi_{\omega},\beta_{-}\varphi_{\omega}):\gamma>1\}$ 
of positive solutions of \eqref{sp} bifurcates from the semi-trivial solution 
$(0,\varphi_{\omega})$ at $\gamma=1$. 
\par \noindent $({\rm 2})$ \hspace{1mm}
When $\kappa>0$, $(\alpha_{-},\beta_{+})\to (0,1)$ as $\gamma\to 1-0$. 
That is, the branch 
$\{(\alpha_{-}\varphi_{\omega},\beta_{+}\varphi_{\omega}):\gamma_m<\gamma<1\}$ 
of positive solutions of \eqref{sp} bifurcates from the semi-trivial solution 
$(0,\varphi_{\omega})$ at $\gamma=1$, where 
$\gamma_m=\inf\{\gamma:(\kappa,\gamma)\in \mathcal{S}_{\kappa,\gamma}\}$, 
and it is given by $\gamma_m=0$ if $\kappa>1/\sqrt{2}$, 
and $\gamma_m=\gamma_{+}$ if $0<\kappa\le 1/\sqrt{2}$. 
\end{remark}

We obtain the following stability and instability results 
of standing waves of \eqref{eq:co} associated with Proposition \ref{prop1}. 
Recall that $\varphi_{\omega}$ is the positive radial solution of \eqref{scalar}. 

\begin{theorem}\label{thm2}
Let $N\le 3$ and $(\kappa,\gamma)\in \mathcal{J}_1\cup \mathcal{J}_2$. 
For any $\omega>0$, the standing wave solution 
$G(\omega t)(\alpha_{+}\varphi_{\omega},\beta_{-}\varphi_{\omega})$ 
of \eqref{eq:co} is stable.
\end{theorem}

\begin{theorem}\label{thm3}
Let $N\le 3$ and $(\kappa,\gamma)\in \mathcal{J}_2$. 
For any $\omega>0$, the standing wave solution 
$G(\omega t)(\alpha_{-}\varphi_{\omega},\beta_{+}\varphi_{\omega})$ 
of \eqref{eq:co} is unstable.
\end{theorem}

\begin{remark}
In this paper, we do not study the stability/instability problem 
of $G(\omega t)(\alpha_{0}\varphi_{\omega},\beta_{0}\varphi_{\omega})$ 
for the case $(\kappa,\gamma)\in \mathcal{J}_3$. 
\end{remark}

\begin{remark}
The result for the case $\kappa=1$ in Theorem \ref{thm3} is announced in \cite{oht} 
together with an outline of the proof. 
\end{remark}

We also obtain the stability and instability results 
of semi-tirivial standing wave at the bifurcation point $\gamma=1$. 
The results depend on the sign of $\kappa$. 

\begin{theorem}\label{thm4}
Let $N\le 3$, $\kappa>0$ and $\gamma=1$. For any $\omega>0$, 
the standing wave solution $(0,e^{2i\omega t}\varphi_{\omega})$ 
of \eqref{eq:co} is unstable.
\end{theorem}

\begin{theorem}\label{thm5}
Let $N\le 3$, $\kappa\le 0$ and $\gamma=1$. For any $\omega>0$, 
the standing wave solution $(0,e^{2i\omega t}\varphi_{\omega})$ 
of \eqref{eq:co} is stable.
\end{theorem}

\begin{remark}
The linearized operator $S_{\omega}''(0,\varphi_{\omega})$ 
around the semi-trivial standing wave is independent of $\kappa$
(see \eqref{slr} and \eqref{sli} below). 
Therefore, Theorems \ref{thm4} and \ref{thm5} are never obtained 
from the linearized analysis only. 
The proof of Theorem \ref{thm5} relies on the variational method of Shatah \cite{sha} 
and on the characterization of the ground states in Theorem \ref{thm6} below. 
\end{remark}

\begin{remark}
For the case $\gamma=1$, using the notation in Section \ref{sect2}, 
we have $\mathcal{L}_R\vec v=(L_1v_1,L_2v_2)$ and $\mathcal{L}_I\vec v=(L_{-1}v_1,L_1v_2)$, 
and the kernel of $S_{\omega}''(0,\varphi_{\omega})$ contains 
a nontrivial element $(\varphi_{\omega},0)$ other than 
the elements $\nabla (0,\varphi_{\omega})$ and $J(0,\varphi_{\omega})$ 
naturally coming from the symmetries of $S_{\omega}$ 
(see \eqref{Sker} below). 
\end{remark}

Next, we consider the ground state problem for \eqref{sp}. We define 
\begin{equation}\label{kappac}
\kappa_c(\gamma)=\frac{1}{2}(\gamma+2)\sqrt{1-\gamma}, \quad 0<\gamma<1.
\end{equation}
Then, $\kappa_c$ is strictly decreasing on the open interval $]0,1[$, 
$\kappa_c(0)=1$ and $\kappa_c(1)=0$. 
We define a function $\gamma_c$ on $]0,1[$ by the inverse function of $\kappa_c$. 
For the ground state problem, it is convenient to divide the parameter domain 
$\mathcal{D}=\{(\kappa,\gamma):\kappa\in \R,~ \gamma>0\}$ into the following sets 
(see Figure \ref{figure:K}). 
\begin{align*}
&\mathcal{K}_1=\{(\kappa,\gamma):\kappa\le 0,~ \gamma>1\}\cup 
\{(\kappa,\gamma):\kappa\ge 1,~ \gamma>0\} \\
&\hspace{12mm} 
\cup \{(\kappa,\gamma):0<\kappa<1,~ \gamma>\gamma_c(\kappa)\}, \\
&\mathcal{K}_2=\{(\kappa,\gamma):\kappa\le 0,~ 0<\gamma\le 1\}\cup 
\{(\kappa,\gamma):0<\kappa<1,~ 0<\gamma<\gamma_c(\kappa)\}, \\
&\mathcal{K}_3=\{(\kappa,\gamma):0<\kappa<1,~ \gamma=\gamma_c(\kappa)\}.
\end{align*}

\begin{figure}[htbp]
\begin{center}
\scalebox{0.4}{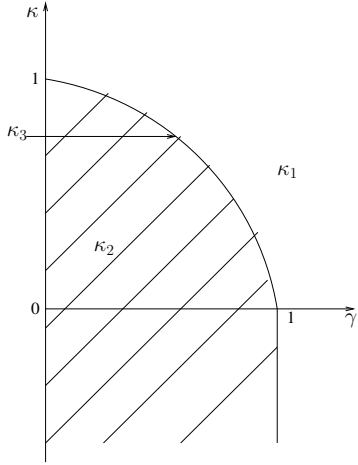 }
\caption{The sets $\mathcal{K}_1$, $\mathcal{K}_2$ and $\mathcal{K}_3$}
\label{figure:K}
\end{center}
\end{figure}

Note that $\mathcal{K}_1$, $\mathcal{K}_2$ and $\mathcal{K}_3$ are mutually disjoint, 
and $\mathcal{D}=\mathcal{K}_1\cup \mathcal{K}_2\cup \mathcal{K}_3$. 
Remark also that since $\sqrt{2\gamma (1-\gamma)}<\kappa_c(\gamma)$ 
for $0<\gamma<1$, we have $\mathcal{J}_0\subset \mathcal{K}_2$. 

Moreover, we define 
\begin{align*}
&\mathcal{G}_{\omega}^0
=\{G(\theta)\tau_y(0,\varphi_{\omega}):\theta\in \R,~ y\in \R^N\}, \\
&\mathcal{G}_{\omega}^1
=\{G(\theta)\tau_y(\alpha_{+}\varphi_{\omega}, \beta_{-}\varphi_{\omega}):
\theta\in \R,~ y\in \R^N\}. 
\end{align*}
Then, the set $\mathcal{G}_{\omega}$ of the ground states for \eqref{sp} 
is determined as follows. 

\begin{theorem}\label{thm6}
Let $N\le 3$ and $\omega>0$. 
\par \noindent $({\rm 1})$ \hspace{1mm}
If $(\kappa,\gamma)\in \mathcal{K}_1$, then 
$\mathcal{G}_{\omega}=\mathcal{G}_{\omega}^1$. 
\par \noindent $({\rm 2})$ \hspace{1mm}
If $(\kappa,\gamma)\in \mathcal{K}_2$, then 
$\mathcal{G}_{\omega}=\mathcal{G}_{\omega}^0$. 
\par \noindent $({\rm 3})$ \hspace{1mm}
If $(\kappa,\gamma)\in \mathcal{K}_3$, then 
$\mathcal{G}_{\omega}=\mathcal{G}_{\omega}^0\cup \mathcal{G}_{\omega}^1$. 
\end{theorem}

The rest of the paper is organized as follows. 
In Section \ref{sect2}, we study some spectral properties 
of the linearized operators around standing waves, 
which are needed in Sections \ref{sect3} and \ref{sect4}. 
In Section \ref{sect3}, we prove Theorems \ref{thm2} and \ref{thm3}, 
while Section \ref{sect4} is devoted to the proof of Theorem \ref{thm4}. 
In Section \ref{sect5}, we study the ground state problem for \eqref{sp}, 
and prove Theorem \ref{thm6}. 
Finally, Theorem \ref{thm5} is proved as a corollary of Theorem \ref{thm6}. 

\section{Linearized Operators}\label{sect2}

In this section, we study spectral properties 
of the linearized operator $S_{\omega}''(\Phi)$. 
Here and hereafter, for $\alpha\ge 0$ and $\beta>0$, we put 
$$\Phi=(\alpha \varphi_{\omega},\beta \varphi_{\omega}), \quad 
\Phi_1=(-\beta \varphi_{\omega},\alpha \varphi_{\omega}), \quad 
\Phi_2=(\alpha \varphi_{\omega},2\beta \varphi_{\omega}).$$
First, by direct computations, we have 
\begin{equation}\label{slri}
\dual{S_{\omega}''(\Phi)\vec u}{\vec u}
=\dual{\mathcal{L}_R \Re \vec u}{\Re \vec u}
+\dual{\mathcal{L}_I \Im \vec u}{\Im \vec u}
\end{equation}
for $\vec u=(u_1,u_2)\in X$, 
where $\Re \vec u=(\Re u_1,\Re u_2)$, $\Im \vec u=(\Im u_1,\Im u_2)$, and 
\begin{align}
&\mathcal{L}_R=\left[\begin{array}{cc}
-\Delta+\omega & 0 \\
0 & -\Delta+\omega
\end{array}\right]
-\left[\begin{array}{cc}
(2\alpha+\gamma \beta) \varphi_{\omega} & \gamma \alpha \varphi_{\omega} \\
\gamma \alpha \varphi_{\omega} & 2\beta \varphi_{\omega}
\end{array}\right], 
\label{slr}\\
&\mathcal{L}_I=\left[\begin{array}{cc}
-\Delta+\omega & 0 \\
0 & -\Delta+\omega
\end{array}\right]
-\left[\begin{array}{cc}
(\alpha-\gamma \beta) \varphi_{\omega} & \gamma \alpha \varphi_{\omega} \\
\gamma \alpha \varphi_{\omega} & \beta \varphi_{\omega}
\end{array}\right]. \label{sli}
\end{align}
Since $S_{\omega}'(G(\theta)\tau_y\Phi)=0$ 
for $y\in \R^N$ and $\theta \in \R$, we see that 
\begin{equation}\label{Sker}
\nabla \Phi\in \ker \mathcal{L}_R, \quad \Phi_2\in  \ker \mathcal{L}_I. 
\end{equation}
For $a\in \R$, we define $L_a$ by 
$$L_av=-\Delta v+\omega v-a\varphi_{\omega} v, \quad v\in H^1(\R^N,\R).$$
We recall some known results on $L_a$. 

\begin{lemma}\label{lem:pos}
Let $N\le 3$ and let $\varphi_{\omega}$ be 
the positive radial solution of \eqref{scalar}.
\par \noindent $({\rm 1})$ \hspace{1mm}
$L_2$ has one negative eigenvalue, 
$\ker L_2$ is spanned by $\{\nabla \varphi_{\omega}\}$, 
and there exists a constant $c_1>0$ such that $\dual{L_2v}{v}\ge c_1 \|v\|_{H^1}^2$ 
for all $v\in H^1(\R^N,\R)$ satisfying 
$(v,\varphi_{\omega})_{L^2}=0$ and $(v,\nabla \varphi_{\omega})_{L^2}=0$. 
\par \noindent $({\rm 2})$ \hspace{1mm}
$L_1$ is non-negative, $\ker L_1$ is spanned by $\{\varphi_{\omega}\}$, 
and there exists $c_2>0$ such that $\dual{L_1v}{v} \ge c_2 \|v\|_{H^1}^2$ 
for all $v\in H^1(\R^N,\R)$ satisfying $(v,\varphi_{\omega})_{L^2}=0$. 
\par \noindent $({\rm 3})$ \hspace{1mm}
If $a<1$, then $L_a$ is positive on $H^1(\R^N,\R)$. 
\par \noindent $({\rm 4})$ \hspace{1mm}
If $1<a<2$, then $\dual{L_a \varphi_{\omega}}{\varphi_{\omega}}<0$, 
and there exists $c_4>0$ such that $\dual{L_a v}{v}\ge c_4 \|v\|_{H^1}^2$ 
for all $v\in H^1(\R^N,\R)$ satisfying $(v,\varphi_{\omega})_{L^2}=0$. 
\end{lemma}

\begin{proof}
Parts (1) and (2) are well-known (see \cite{wei1}). 
Note that the quadratic nonlinearity in \eqref{scalar} 
is $L^2$-subcritical if and only if $N\le 3$, 
and that the assumption $N\le 3$ is essential for (1). 
Parts (3) and (4) follow from (1) and (2) immediately. 
\end{proof}

In the next lemma, we give the diagonalization of $\mathcal{L}_R$ and $\mathcal{L}_I$.

\begin{lemma}\label{diag}
By orthogonal matrices 
$$A=\frac{1}{\sqrt{\alpha^2+\beta^2}}
\left[\begin{array}{cc}
\alpha & \beta \\
-\beta & \alpha 
\end{array}\right], \quad 
B=\frac{1}{\sqrt{\alpha^2+4\beta^2}}
\left[\begin{array}{cc}
\alpha & 2\beta \\
-2\beta & \alpha 
\end{array}\right],$$
$\mathcal{L}_R$ and $\mathcal{L}_I$ are diagonalized as follows: 
$$\mathcal{L}_R=A^*
\left[\begin{array}{cc}
L_2 & 0 \\
0 & L_{(2-\gamma)\beta}
\end{array}\right]A, \quad 
\mathcal{L}_I=B^*
\left[\begin{array}{cc}
L_1 & 0 \\
0 & L_{(1-2\gamma)\beta}
\end{array}\right]B.$$
\end{lemma}

\begin{proof}
The computation is straightforward, and we omit the details.
\end{proof}

The next three lemmas establish the coercivity properties 
of the operators $\mathcal{L}_R$ and $\mathcal{L}_I$. 
They represent the main results of this section, 
and are the key points in the proofs of Theorems \ref{thm2} and \ref{thm3}.

\begin{lemma}\label{Lr1}
If $(2-\gamma)\beta<1$, then there exists a constant $\delta_1>0$ such that 
$\dual{\mathcal{L}_R \vec v}{\vec v}\ge \delta_1 \|\vec v\|_{X}^2$
for all $\vec v\in H^1(\R^N,\R)^2$ satisfying 
$(\vec v,\Phi)_{H}=0$ and $(\vec v,\nabla \Phi)_{H}=0$. 
\end{lemma}

\begin{proof}
By Lemma \ref{diag}, we have 
$\dual{\mathcal{L}_R \vec v}{\vec v}
=\dual{L_2w_1}{w_1}+\dual{L_{(2-\gamma)\beta}w_2}{w_2}$, 
where $\vec w=A\vec v$. Since we have 
$$(w_1,\varphi_{\omega})_{L^2}
=\frac{(\vec v, \Phi)_{H}}{\sqrt{\alpha^2+\beta^2}}=0, \quad 
(w_1,\nabla \varphi_{\omega})_{L^2}
=\frac{(\vec v,\nabla \Phi)_{H}}{\sqrt{\alpha^2+\beta^2}}=0,$$
it follows from Lemma \ref{lem:pos} (1) that 
$\dual{L_2 w_1}{w_1}\ge c_1 \|w_1\|_{H^1}^2$. 
Moreover, by the assumption $(2-\gamma)\beta<1$ 
and by Lemma \ref{lem:pos} (3), we have 
$\dual{L_{(2-\gamma)\beta}w_2}{w_2}\ge c_3 \|w_2\|_{H^1}^2$. 
This completes the proof. 
\end{proof}

\begin{lemma}\label{Lr2}
If $1\le (2-\gamma)\beta<2$, then there exists a constant $\delta_2>0$ such that 
$\dual{\mathcal{L}_R \vec v}{\vec v}\ge \delta_2 \|\vec v\|_{X}^2$
for all $\vec v\in H^1_{\rm rad}(\R^N,\R)^2$ satisfying 
$(\vec v,\Phi)_{H}=0$ and $(\vec v,\Phi_1)_{H}=0$, where 
$\Phi_1=(-\beta \varphi_{\omega},\alpha \varphi_{\omega})$. 
\end{lemma}

\begin{proof}
By Lemma \ref{diag}, we have 
$\dual{\mathcal{L}_R \vec v}{\vec v}=\dual{L_2 w_1}{w_1}
+\dual{L_{(2-\gamma)\beta}w_2}{w_2}$, 
where $\vec w=A\vec v$. 
Then we have $(w_1,\varphi_{\omega})_{L^2}
=(\vec v, \Phi)_{H}/\sqrt{\alpha^2+\beta^2}=0$. 
Moreover, since $\varphi_{\omega}$ and $w_1$ are radially symmetric, 
we have $(w_1,\nabla \varphi_{\omega})_{L^2}=0$. 
Thus, it follows from Lemma \ref{lem:pos} (1) that 
$\dual{L_2 w_1}{w_1}\ge c_1\|w_1\|_{H^1}^2$. 
Moreover, since $(w_2,\varphi_{\omega})_{L^2}
=(\vec v, \Phi_1)_{H}/\sqrt{\alpha^2+\beta^2}=0$, 
it follows from the assumption $1\le (2-\gamma)\beta<2$ 
and Lemma \ref{lem:pos} (2), (4) that 
$\dual{L_{(2-\gamma)\beta} w_2}{w_2}\ge c_2\|w_2\|_{H^1}^2$. 
\end{proof}

\begin{lemma}\label{Li1}
There exists a constant $\delta_3>0$ such that 
$\dual{\mathcal{L}_I \vec v}{\vec v}\ge \delta_3 \|\vec v\|_{X}^2$ 
for all $\vec v\in H^1(\R^N,\R)^2$ satisfying $(\vec v,\Phi_2)_{H}=0$, 
where $\Phi_2=(\alpha \varphi_{\omega},2\beta \varphi_{\omega})$. 
\end{lemma}

\begin{proof}
By Lemma \ref{diag}, we have 
$\dual{\mathcal{L}_I \vec v}{\vec v}=\dual{L_1w_1}{w_1}
+\dual{L_{(1-2\gamma) \beta}w_2}{w_2}$, 
where $\vec w=B\vec v$. 
Since $(w_1,\varphi_{\omega})_{L^2}
=(\vec v,\Phi_2)_{H}/\sqrt{\alpha^2+4\beta^2}=0$, 
Lemma \ref{lem:pos} (2) implies 
$\dual{L_1 w_1}{w_1}\ge c_2\|w_1\|_{H^1}^2$. 
Moreover, since $(1-2\gamma) \beta<1$, 
it follows from Lemma \ref{lem:pos} (3) that 
$\dual{L_{(1-2\gamma) \beta}w_2}{w_2}\ge c_3\|w_2\|_{H^1}^2$. 
\end{proof}

The last two lemmas of this section make connections 
between parameters $(\kappa,\gamma)$ and the criteria used in 
Lemma \ref{Lr1}, \ref{Lr2} and \ref{Li1} on $\beta$. 

\begin{lemma}\label{lem:gb1}
Let $(\kappa,\gamma)\in \mathcal{J}_1\cup \mathcal{J}_2$. 
Then, $(2-\gamma)\beta_{-}<1$ and $(1-2\gamma)\beta_{-}<1$. 
\end{lemma}

\begin{proof}
We put $D=\kappa^2+2\gamma (\gamma-1)$. 
By the second equation of \eqref{eq:ab}, we have $0<\beta_{-}<1$. 
Thus, we have $(1-2\gamma)\beta_{-}<\beta_{-}<1$. 
If $\gamma>1$, then $(2-\gamma)\beta_{-}<\beta_{-}<1$. 
While, if $0<\gamma\le 1$, then $\kappa>0$, $D>0$ and $(2-\gamma)\beta_{-}
<(2-\gamma)(\kappa^2+\gamma^2)/(2\kappa^2+\gamma^3)<1$. 
Note that the last inequality is equivalent to $D>0$. 
\end{proof}

\begin{lemma}\label{lem:gb2}
Let $(\kappa,\gamma)\in \mathcal{J}_2$. 
Then, $1<(2-\gamma)\beta_{+}<2$ and $(1-2\gamma)\beta_{+}<1$. 
\end{lemma}

\begin{proof}
We put $D=\kappa^2+2\gamma (\gamma-1)$. 
Since $0<\beta_{+}<1$, we have 
$(2-\gamma)\beta_{+}<2\beta_{+}<2$ and $(1-2\gamma)\beta_{+}<\beta_{+}<1$. 
Next, we see that $(2-\gamma)\beta_{+}>1$ is equivalent to 
$(2-\gamma) \kappa>\gamma \sqrt{D}$. 
Since $0<\gamma<1$ and $\kappa>0$, we have 
$\gamma \sqrt{D}<\gamma \kappa<(2-\gamma) \kappa$. 
\end{proof}

\begin{remark}
When $(\kappa,\gamma)\in \mathcal{J}_3$, 
we have $D=\kappa^2+2\gamma (\gamma-1)=0$, 
$(2-\gamma)\beta_{0}=1$ and $(1-2\gamma)\beta_{0}<1$. 
\end{remark}

\section{Proofs of Theorems \ref{thm2} and \ref{thm3}} \label{sect3}

In this section we prove Theorems \ref{thm2} and \ref{thm3} 
using the results of Section \ref{sect2} and the following propositions. 
Proposition \ref{prop2} follows from Theorem 3.4 
of Grillakis, Shatah and Strauss \cite{GSS1}
(see also \cite{wei2} and \cite[Section 3]{CCO1}). 
While, Proposition \ref{prop3} follows from Theorem 1 
of \cite{oht} (see also \cite{GSS1,mae,SS}). 

\begin{proposition}\label{prop2}
Let $\vec \phi\in \mathcal{A}_{\omega}$. 
Assume that there exists a constant $\delta>0$ such that 
$\dual{S_{\omega}''(\vec \phi)\vec w}{\vec w}\ge \delta \|\vec w\|_{X}^2$ 
for all $\vec w\in X$ satisfying 
$(\vec \phi,\vec w)_{H}=(J\vec \phi,\vec w)_{H}=0$ 
and $(\nabla \vec \phi,\vec w)_{H}=0$. 
Then the standing wave solution $G(\omega t)\vec \phi$ 
of \eqref{eq:co} is stable. 
\end{proposition}

\begin{proposition}\label{prop3}
Let $\vec \phi\in \mathcal{A}_{\omega}$ be radially symmetric. 
Assume that there exist $\vec \psi\in X_{\rm rad}$ 
and a constant $\delta>0$ such that $\|\vec \psi\|_{H}=1$, 
$(\vec \psi,\vec \phi)_{H}=(\vec \psi,J\vec \phi)_{H}=0$, 
$\dual{S_{\omega}''(\vec \phi)\vec \psi}{\vec \psi}<0$, and 
$\dual{S_{\omega}''(\vec \phi)\vec w}{\vec w}\ge \delta \|\vec w\|_{X}^2$ 
for all $\vec w\in  X_{\rm rad}$ satisfying 
$(\vec \phi,\vec w)_{H}=(J\vec \phi,\vec w)_{H}=(\vec \psi,\vec w)_{H}=0$. 
Then the standing wave solution $G(\omega t)\vec \phi$ 
of \eqref{eq:co} is unstable. 
\end{proposition}

\begin{proof}[Proof of Theorem \ref{thm2}]
For $(\kappa,\gamma)\in \mathcal{J}_1\cup\mathcal{J}_2$, 
let $(\alpha,\beta)=(\alpha_{+},\beta_{-})$. Let $\vec w\in X$ satisfy 
$(\Phi,\vec w)_{H}=(J\Phi,\vec w)_{H}=0$ and $(\nabla \Phi,\vec w)_{H}=0$.  
By \eqref{slri}, we have 
$$\dual{S_{\omega}''(\Phi)\vec w}{\vec w}
=\dual{\mathcal{L}_R \Re \vec w}{\Re \vec w}
+\dual{\mathcal{L}_I \Im \vec w}{\Im \vec w}.$$
Since $(\Phi,\Re \vec w)_{H}=(\Phi,\vec w)_{H}=0$ and 
$(\nabla \Phi,\Re \vec w)_{H}=(\nabla \Phi,\vec w)_{H}=0$, 
it follows from Lemmas \ref{lem:gb1} and \ref{Lr1} that 
$\dual{\mathcal{L}_R \Re \vec w}{\Re \vec w}\ge \delta_1\|\Re \vec w\|_{X}^2$. 
While, since $(\Im \vec w,\Phi_2)_{H}=(J\Phi,\vec w)_{H}=0$, 
Lemma \ref{Li1} implies $\dual{\mathcal{L}_I \Im \vec w}{\Im \vec w}
\ge \delta_3\|\Im \vec w\|_{X}^2$. 
Therefore, Theorem \ref{thm2} follows from Proposition \ref{prop2}. 
\end{proof}

\begin{proof}[Proof of Theorem \ref{thm3}]
For $(\kappa,\gamma)\in \mathcal{J}_2$, let $(\alpha,\beta)=(\alpha_{-},\beta_{+})$. 
We take $\vec \psi=\Phi_1/\|\Phi_1\|_{H}$. 
Then we have $\|\vec \psi\|_{H}=1$, $(\vec \psi,\Phi)_{H}=0$ 
and $(\vec \psi,J\Phi)_{H}=0$. 
Moreover, by Lemma \ref{lem:gb2} and Lemma \ref{lem:pos} (4), we have 
$$\dual{S_{\omega}''(\Phi)\vec \psi}{\vec \psi}
=\dual{\mathcal{L}_R \vec \psi}{\vec \psi}
=\dual{L_{(2-\gamma)\beta} \varphi_{\omega}}{\varphi_{\omega}}
/{\|\varphi_{\omega}\|_{L^2}^2}<0.$$
Finally, let $\vec w\in  X_{\rm rad}$ satisfy 
$(\Phi,\vec w)_{H}=(J\Phi,\vec w)_{H}=(\vec \psi,\vec w)_{H}=0$. 
Since $(\Phi,\Re \vec w)_{H}=(\Phi,\vec w)_{H}=0$ and 
$(\Phi_1,\Re \vec w)_{H}=(\Phi_1,\vec w)_{H}=0$, 
by Lemmas \ref{lem:gb2} and \ref{Lr2}, we have 
$\dual{\mathcal{L}_R \Re \vec w}{\Re \vec w}\ge \delta_2\|\Re \vec w\|_{X}^2$. 
While, since $(\Im \vec w,\Phi_2)_{H}=(J\Phi,\vec w)_{H}=0$, 
by Lemma \ref{Li1}, we have 
$\dual{\mathcal{L}_I \Im \vec w}{\Im \vec w}\ge \delta_3\|\Im \vec w\|_{X}^2$. 
Thus, by \eqref{slri}, we have 
$\dual{S_{\omega}''(\vec \phi)\vec w}{\vec w}\ge \delta \|\vec w\|_{X}^2$, 
and Theorem \ref{thm3} follows from Proposition \ref{prop3}. 
\end{proof}

\section{Proof of Theorem \ref{thm4}} \label{sect4}

We introduce the following Proposition \ref{prop4} to prove Theorem \ref{thm4}. 
It is a modification of Theorem 2 of \cite{oht}. 
In what follows, $\sgn(\mu)$ denotes the sign of any real $\mu$. 

\begin{proposition}\label{prop4}
Let $\vec \phi\in \mathcal{A}_{\omega}$ be radially symmetric. 
Assume that there exist $\vec \psi\in X_{\rm rad}$ such that 
\par \noindent $({\rm i})$ \hspace{1mm} 
$\|\vec \psi\|_{H}=1$, $(\vec \psi,\vec \phi)_{H}=0$, 
$(\vec \psi,J\vec \phi)_{H}=(\vec \psi,J\vec \phi)_{X}=0$, 
$S_{\omega}''(\vec \phi)\vec \psi=\vec 0$, 
\par \noindent $({\rm ii})$ \hspace{1mm} 
there exists a positive constant $k_0$ such that 
$\dual{S_{\omega}''(\vec \phi)\vec w}{\vec w}\ge k_0\|\vec w\|_{X}^2$ 
\par \hspace{2mm}
for all $\vec w\in  X_{\rm rad}$ satisfying 
$(\vec \phi,\vec w)_{H}=(J\vec \phi,\vec w)_{H}=
(\vec \psi,\vec w)_{H}=0$,
\par \noindent $({\rm iii})$ \hspace{1mm} 
there exist positive constants $k_1$, $k_2$ and $\varepsilon$ such that 
$$\sgn (\lambda) \cdot 
\dual{S_{\omega}'(\vec \phi+\lambda \vec \psi+\vec z)}{\vec \psi}
\le -k_1 \lambda^2+k_2 \|\vec z\|_{X}^2+o(\lambda^2+\|\vec z\|_{X}^2)$$
\par \hspace{1mm}
for all $\lambda\in \R$ and $\vec z\in X_{\rm rad}$ 
satisfying $|\lambda|+\|\vec z\|_{X}<\varepsilon$. 
\par \noindent 
Then the standing wave solution $G(\omega t)\vec \phi$ of \eqref{eq:co} is unstable. 
\end{proposition}

We first prove Theorem \ref{thm4} using Proposition \ref{prop4}. 

\begin{proof}[Proof of Theorem \ref{thm4}]
The proof consists of verifying the assumptions (i), (ii), (iii) of Proposition \ref{prop4}. 
Let $(\alpha,\beta)=(0,1)$ and $\Phi=(0,\varphi_{\omega})$. We take 
$$\vec \psi=(\psi_1,\psi_2)=(\varphi_{\omega},0)/\|\varphi_{\omega}\|_{L^2}.$$
Then, $\|\vec \psi\|_{H}=1$, $(\vec \psi,\Phi)_{H}=0$, 
$(\vec \psi,J\Phi)_{H}=(\vec \psi,J\Phi)_{X}=0$, and 
$$S_{\omega}''(\vec \phi)\vec \psi
=(L_1\varphi_{\omega},0)/\|\varphi_{\omega}\|_{L^2}=(0,0).$$
Thus, (i) is satisfied. 
The assumption (ii) is proved in the same way as the proof of Theorem \ref{thm3}. 
Finally, we prove (iii). 
Let $\lambda\in \R$ and $\vec z=(z_1,z_2)\in X_{\rm rad}$, and put 
$\vec v=(v_1,v_2)=\lambda \vec \psi+\vec z$. 
Then, we have 
$$v_1=\lambda \psi_1+z_1, \quad v_2=z_2, \quad 
\psi_1=\varphi_{\omega}/\|\varphi_{\omega}\|_{L^2},$$ 
and 
\begin{align*}
&\|\varphi_{\omega}\|_{L^2}\dual{S_{\omega}'(\Phi+\vec v)}{\vec \psi} \\
&=\Re \int_{\R^N}\{\nabla v_1\cdot \nabla \varphi_{\omega}
+\omega v_1\varphi_{\omega}-\kappa |v_1|v_1\varphi_{\omega}
-v_1(\varphi_{\omega}+\overline{v_2})\varphi_{\omega}\}\,dx \\
&=\Re \int_{\R^N}\{v_1(-\Delta \varphi_{\omega}
+\omega \varphi_{\omega}-\varphi_{\omega}^2)
-\kappa |v_1|v_1\varphi_{\omega}-v_1\overline{v_2}\varphi_{\omega}\}\,dx \\
&=-\kappa \Re \int_{\R^N}|v_1|v_1\varphi_{\omega}\,dx
-\Re \int_{\R^N}v_1\overline{v_2}\varphi_{\omega}\,dx. 
\end{align*}
Thus, we have 
\begin{equation}\label{mmm}
\dual{S_{\omega}'(\Phi+\vec v)}{\vec \psi} 
=-\kappa \Re \int_{\R^N}|v_1|v_1\psi_1\,dx
-\Re \int_{\R^N}v_1\overline{v_2}\psi_1\,dx.
\end{equation}
Here, we have
$$\sgn (\lambda) \cdot \kappa \Re \int_{\R^N}|\lambda \psi_1|\lambda \psi_1 \psi_1\,dx
=C_0\lambda^2, \quad \text{where }
C_0:=\kappa \|\varphi_{\omega}\|_{L^3}^3/\|\varphi_{\omega}\|_{L^2}^3,$$
and the first term of the right hand side of \eqref{mmm} is estimated as follows. 
\begin{align*}
&\left|\sgn (\lambda)\cdot \kappa \Re \int_{\R^N}|v_1|v_1\psi_1\,dx-C_0\lambda^2\right|
\le \kappa \int_{\R^N}\Big||v_1|v_1-|\lambda \psi_1|\lambda \psi_1 \Big|\psi_1\,dx \\
&\le C \int_{\R^N}(|v_1|+|\lambda \psi_1|)|v_1-\lambda \psi_1|\psi_1\,dx 
\le C \int_{\R^N}(|\lambda \psi_1|+|z_1|)|z_1|\psi_1\,dx \\
&\le C|\lambda|\|z_1\|_{L^3}\|\psi_1\|_{L^3}^2
+C\|z_1\|_{L^3}^2\|\psi_1\|_{L^3} \le C_0\lambda^2/4+C_1\|z_1\|_{H^1}^2
\end{align*}
for some constant $C_1$ depending on $\varphi_\omega$. 
Here, in the last inequality, 
we used the inequality of the type $2ab\le \varepsilon^2a^2+b^2/\varepsilon^2$. 
While, the second term of the right hand side of \eqref{mmm} is estimated as follows. 
\begin{align*}
&\left|\Re \int_{\R^N}v_1\overline{v_2}\psi_1\,dx \right|
\le |\lambda|\|z_2\|_{L^3}\|\psi_1\|_{L^3}^2
+\|z_1\|_{L^3}\|z_2\|_{L^3}\|\psi_1\|_{L^3} \\
&\le C_0\lambda^2/4+C_2\|z_2\|_{H^1}^2+C_3\|z_1\|_{H^1}\|z_2\|_{H^1}
\end{align*}
for some positive constants $C_2$ and $C_3$. Thus, we have 
$$\sgn \lambda \cdot \dual{S_{\omega}'(\Phi+\lambda \vec \psi+\vec z)}{\vec \psi}
\le -C_0\lambda^2/2+C_4\|\vec z\|_{X}^2$$
for some constant $C_4>0$. This completes the proof. 
\end{proof}

In the rest of this section, we give the proof of Proposition \ref{prop4} 
by modifying the proof of Theorem 2 of \cite{oht}. 
We define 
$$\mathcal{N}_{\varepsilon}(\vec \phi)=\{\vec u\in X_{\rm rad}:
\inf_{\theta\in \R}\|G(\theta)\vec u-\vec \phi\|_{X}<\varepsilon\},$$
and the identification operator $I:X\to X^*$ by 
$$\dual{I\vec u}{\vec v}=(\vec u,\vec v)_{H}, \quad \vec u, \vec v\in X.$$

\begin{lemma}\label{lem:ift}
There exist $\varepsilon>0$ and a $C^2$ map 
$\Theta:\mathcal{N}_{\varepsilon}(\vec \phi)\to \R/2\pi \Z$ such that 
\begin{align}
&\|G(\Theta (\vec u))\vec u-\vec \phi\|_{X}
\le \|G(\theta)\vec u-\vec \phi\|_{X}, 
\nonumber \\
&(G(\Theta (\vec u))\vec u,J\vec \phi)_{X}=0, \quad 
\Theta (G(\theta)\vec u)=\Theta (\vec u)-\theta,
\nonumber \\
&I^{-1}\Theta'(\vec u)=
\frac{JG(-\Theta(\vec u))(1-\Delta)\vec \phi}
{(G(\Theta(\vec u))\vec u,J^2\vec \phi)_X}
\label{ift1}
\end{align}
for all $\vec u\in \mathcal{N}_{\varepsilon}(\vec \phi)$ 
and $\theta \in \R/2\pi \Z$. 
\end{lemma}

\begin{proof}
See Lemma 3.2 of \cite{GSS1}. 
Note that $\vec \phi\in H^3(\R^N)^2$
by the elliptic regularity for \eqref{sp}. 
\end{proof}

We put $M(\vec u)=G(\Theta (\vec u))\vec u$. 
Then we have $M(\vec \phi)=\vec \phi$ and 
$M(G(\theta)\vec u)=M(\vec u)$ 
for $\vec u\in \mathcal{N}_{\varepsilon}(\vec \phi)$ 
and $\theta \in \R$. 
We define $\mathcal{A}$ and $\Lambda$ by 
\begin{equation}\label{def:alam}
\mathcal{A}(\vec u)=(M(\vec u),J^{-1}\vec \psi)_{H}, \quad 
\Lambda (\vec u)=(M(\vec u),\vec \psi)_{H}
\end{equation}
for $\vec u\in \mathcal{N}_{\varepsilon}(\vec \phi)$. 
Then we have 
\begin{align}
&JI^{-1}\mathcal{A}'(\vec u)=G(-\Theta (\vec u))\vec \psi
-\Lambda (\vec u)JI^{-1}\Theta'(\vec u), \label{ift2} \\
&0=\frac{d}{d\theta}\mathcal{A}(G(\theta)\vec u)|_{\theta=0}
=\dual{\mathcal{A}'(\vec u)}{J\vec u} 
=-\dual{I\vec u}{JI^{-1}\mathcal{A}'(\vec u)}. \label{ift3}
\end{align}
We define $\mathcal{P}$ by 
$$\mathcal{P}(\vec u)
=\dual{E'(\vec u)}{JI^{-1}\mathcal{A}'(\vec u)}$$
for $\vec u\in \mathcal{N}_{\varepsilon}(\vec \phi)$. 
Note that by \eqref{ift1}, \eqref{ift2} and \eqref{ift3}, we have 
\begin{align}
\mathcal{P}(\vec u)
&=\dual{S_{\omega}'(\vec u)}{JI^{-1}\mathcal{A}'(\vec u)} 
\nonumber \\
&=\dual{S_{\omega}'(M(\vec u))}{\vec \psi}
-\frac{\Lambda (\vec u)}{(M(\vec u),J^2\vec \phi)_{X}}
\dual{S_{\omega}'(M(\vec u))}{J^2(1-\Delta)\vec \phi}.
\label{ift4}
\end{align}

\begin{lemma}\label{lem:ap}
Let $\mathcal{I}$ be an interval of $\R$. 
Let $\vec u\in C(\mathcal{I},X)\cap C^1(\mathcal{I},X^*)$ 
be a solution of \eqref{eq:co}, and assume that 
$\vec u(t)\in \mathcal{N}_{\varepsilon}(\vec \phi)$ 
for all $t\in \mathcal{I}$. Then 
$$\frac{d}{dt}\mathcal{A}(\vec u(t))
=\mathcal{P}(\vec u(t)), \quad t\in \mathcal{I}.$$
\end{lemma}

\begin{proof}
See Lemma 4.6 of \cite{GSS1} and Lemma 2 of \cite{oht}. 
\end{proof}

\begin{lemma}\label{lem:elp}
There exist positive constants $k^*$ and $\varepsilon_0$  such that 
$$E(\vec u)\ge E(\vec \phi)
+k^* \sgn \Lambda (\vec u) \cdot \mathcal{P}(\vec u)$$
for all $\vec u\in \mathcal{N}_{\varepsilon_0}(\vec \phi)$ 
satisfying $Q(\vec u)=Q(\vec \phi)$. 
\end{lemma}

\begin{proof}
We put $\vec v=M(\vec u)-\vec \phi$, 
and decompose $\vec v$ as 
$$\vec v=a\vec \phi+bJ\vec \phi+c\vec \psi+\vec w,$$
where $a$, $b$, $c\in \R$, 
and $\vec w\in X_{\rm rad}$ satisfies 
$(\vec \phi,\vec w)_{H}=(J\vec \phi,\vec w)_{H}
=(\vec \psi,\vec w)_{H}=0$. 
Since $Q(\vec \phi)=Q(\vec u)=Q(M(\vec u))
=Q(\vec \phi)+(\vec \phi,\vec v)_{H}+Q(\vec v)$ 
and $(\vec \phi,\vec v)_{H}=a\|\vec \phi\|_{H}^2$,
we have $a=O(\|\vec v\|_{X}^2)$. 
Moreover, by Lemma \ref{lem:ift} and 
by the assumption (i) of Proposition \ref{prop4}, 
we have $(M(\vec u),J\vec \phi)_{X}=0$ 
and $(J\vec \phi,\vec \psi)_{X}=0$. 
Thus, we have $0=(\vec v,J\vec \phi)_{X}
=b\|J\vec \phi\|_{X}^2+(\vec w,J\vec \phi)_{X}$, 
$|b| \|J\vec \phi\|_{X}\le \|\vec w\|_{X}$ and 
\begin{equation}\label{cvw2}
\|\vec v\|_{X}\le 
|c|\|\vec \psi\|_{X}+2\|\vec w\|_{X}+O(\|\vec v\|_{X}^2).
\end{equation}
Since $S_{\omega}'(\vec \phi)=0$ and $Q(\vec u)=Q(\vec \phi)$, 
by the Taylor expansion, we have 
\begin{equation}\label{tay1}
E(\vec u)-E(\vec \phi)
=S_{\omega}(M(\vec u))-S_{\omega}(\vec \phi) 
=\frac{1}{2}\dual{S_{\omega}''(\vec \phi)\vec v}{\vec v}
+o(\|\vec v\|_{X}^2). 
\end{equation}
Here, since $a=O(\|\vec v\|_{X}^2)$ and 
$S_{\omega}''(\vec \phi)(J\vec \phi)=S_{\omega}''(\vec \phi)\vec \psi=\vec 0$, 
by the assumption (ii) of Proposition \ref{prop4}, we have 
\begin{align}
&E(\vec u)-E(\vec \phi)
=\frac{1}{2}\dual{S_{\omega}''(\vec \phi)\vec v}{\vec v}
+o(\|\vec v\|_{X}^2) 
\nonumber \\
&=\frac{1}{2}\dual{S_{\omega}''(\vec \phi)\vec w}{\vec w}
+o(\|\vec v\|_{X}^2) 
\ge \frac{k_0}{2}\|\vec w\|_{X}^2-o(\|\vec v\|_{X}^2). 
\label{tay21}
\end{align}
On the other hand, we have 
$c=(\vec v,\vec \psi)_{H}=\Lambda (\vec u)=O(\|\vec v\|_{X})$, 
$$S_{\omega}'(\vec \phi+\vec v)
=S_{\omega}'(\vec \phi)+S_{\omega}''(\vec \phi)\vec v+o(\|\vec v\|_{X})
=S_{\omega}''(\vec \phi)\vec w+o(\|\vec v\|_{X}),$$
and $(M(\vec u),J^2\vec \phi)_{X}
=(\vec \phi, J^2\vec \phi)_{X}+O(\|\vec v\|_{X})$.
Thus, by \eqref{ift4} we have 
$$\mathcal{P}(\vec u)
=\dual{S_{\omega}'(\vec \phi+\vec v)}{\vec \psi}
+\frac{c}{\|J\vec \phi\|_{X}^2}
\dual{S_{\omega}''(\vec \phi)\vec w}{J^2(1-\Delta)\vec \phi}
+o(\|\vec v\|_{X}^2).$$
Here, by the assumption (iii) of Proposition \ref{prop4}, 
we have 
\begin{align*}
&\sgn (c) \cdot \dual{S_{\omega}'(\vec \phi+\vec v)}{\vec \psi} \\
&\le -k_1 c^2+k_2 \|a\vec \phi+bJ\vec \phi+\vec w\|_X^2
+o(c^2+\|a\vec \phi+bJ\vec \phi+\vec w\|_X^2) \\
&\le -k_1 c^2+k_3\|\vec w\|_X^2+o(\|\vec v\|_X^2).
\end{align*}
Moreover, we have 
$$\left|\frac{c}{\|J\vec \phi\|_{X}^2}
\dual{S_{\omega}''(\vec \phi)\vec w}{J^2(1-\Delta)\vec \phi}\right|
\le k|c|\|\vec w\|_X
\le \frac{k_1}{2}c^2+k_4\|\vec w\|_X^2.$$
Thus, we have 
\begin{equation}\label{tay22}
-\sgn \Lambda (\vec u) \cdot \mathcal{P}(\vec u)
\ge \frac{k_1}{2} c^2-k_5\|\vec w\|_{X}^2-o(\|\vec v\|_{X}^2)
\end{equation}
with some constant $k_5>0$. 
By \eqref{tay21} and \eqref{tay22}, we have 
\begin{equation}\label{tay23}
E(\vec u)-E(\vec \phi)
-k^*\sgn \Lambda (\vec u)\cdot \mathcal{P}(\vec u)
\ge k_6c^2+k_7\|\vec w\|_{X}^2-o(\|\vec v\|_{X}^2),
\end{equation}
where we put $k^*={k_0}/(4k_5)$, 
$k_6=k^*k_1/2$ and $k_7={k_0}/{4}$. 
Finally, since $\|\vec v\|_X=\|M(\vec u)-\varphi_{\omega}\|_X<\varepsilon_0$, 
it follows from \eqref{cvw2} that 
the right hand side of \eqref{tay23} is non-negative, 
if $\varepsilon_0$ is sufficiently small. 
This completes the proof. 
\end{proof}

\begin{lemma}\label{lem:data}
There exist $\lambda_1>0$ and a continuous curve 
$(-\lambda_1,\lambda_1)\ni \lambda\mapsto \vec \phi_{\lambda}\in X_{\rm rad}$ 
such that $\vec \phi_{0}=\vec \phi$ and 
$$E(\vec \phi_{\lambda})<E(\vec \phi), \quad 
Q(\vec \phi_{\lambda})=Q(\vec \phi), \quad 
\lambda \mathcal{P}(\vec \phi_{\lambda})<0$$
for $0<|\lambda|<\lambda_1$. 
\end{lemma}

\begin{proof}
For $\lambda$ close to $0$, we define 
$$\vec \phi_{\lambda}=\vec \phi+\lambda \vec \psi
+\sigma(\lambda)\vec \phi, \quad \sigma (\lambda)
=\left(1-\frac{Q(\vec \psi)}{Q(\vec \phi)}\lambda^2\right)^{1/2}-1.$$
Then, we have $Q(\vec \phi_{\lambda})=Q(\vec \phi)$, 
$\sigma (\lambda)=O(\lambda^2)$, $\sigma'(\lambda)=O(\lambda)$ and 
\begin{align*}
S_{\omega}(\vec \phi_{\lambda})-S_{\omega}(\vec \phi)
=\int_{0}^{\lambda}\frac{d}{ds}S_{\omega}(\vec \phi_{s})\,ds
=\int_{0}^{\lambda}\dual{S_{\omega}'(\vec \phi_{s})}
{\vec \psi+\sigma'(s)\vec \phi}\,ds.
\end{align*}
Here, by the assumption (iii) of Proposition \ref{prop4}, we have 
$$\sgn (s) \cdot \dual{S_{\omega}'(\vec \phi_{s})}{\vec \psi}\le -k_1s^2+o(s^2).$$
Moreover, since 
$S_{\omega}'(\vec \phi_{s})=S_{\omega}'(\vec \phi)
+S_{\omega}''(\vec \phi)(s\vec \psi+\sigma(s)\vec \phi)+o(s)=o(s)$, 
we have $\dual{S_{\omega}'(\vec \phi_{s})}{\sigma'(s)\vec \phi}=o(s^2)$. 
Thus, we have $S_{\omega}(\vec \phi_{\lambda})-S_{\omega}(\vec \phi)
\le -k_1|\lambda|^3/3+o(\lambda^3)$. 
Finally, by \eqref{tay22}, we have 
$\lambda \mathcal{P}(\vec \phi_{\lambda})
\le -k_1|\lambda|^3/2+o(\lambda^3)$. 
\end{proof}

\begin{proof}[Proof of Proposition \ref{prop4}]
Suppose that $G(\omega t)\vec \phi$ is stable. 
For $\lambda$ close to $0$, 
let $\vec \phi_{\lambda}\in X_{\rm rad}$ 
be the function given in Lemma \ref{lem:data}, 
and let $\vec u_{\lambda}(t)$ be the solution of \eqref{eq:co} 
with $\vec u_{\lambda}(0)=\vec \phi_{\lambda}$. 
Then, there exists $\lambda_0>0$ such that 
if $|\lambda|<\lambda_0$, 
then $\vec u_{\lambda}(t)\in \mathcal{N}_{\varepsilon_0}(\vec \phi)$ 
for all $t\ge 0$, where $\varepsilon_0$ is the positive constant 
given in Lemma \ref{lem:elp}. 
Moreover, by the definition \eqref{def:alam}, 
there exist positive constants $C_1$ and $C_2$ such that 
$|\mathcal{A}(\vec u)|\le C_1$ and $|\Lambda(\vec u)|\le C_2$ 
for all $\vec u\in \mathcal{N}_{\varepsilon_0}(\vec \phi)$. 
Let $-\lambda_0<\lambda<0$ and 
put $\delta_{\lambda}=E(\vec \phi)-E(\vec \phi_{\lambda})$. 
Since $\mathcal{P}(\vec \phi_{\lambda})>0$ and 
$t\mapsto \mathcal{P}(\vec u_{\lambda}(t))$ is continuous, 
by Lemma \ref{lem:elp} and by the conservation laws of $E$ and $Q$, 
we see that $\mathcal{P}(\vec u_{\lambda}(t))>0$ for all $t\ge 0$ and 
$$\delta_{\lambda}
=E(\vec \phi)-E(\vec u_{\lambda}(t))
\le -k^*\sgn \Lambda (\vec u_{\lambda}(t)) \cdot 
\mathcal{P}(\vec u_{\lambda}(t))
\le k^*C_2 \mathcal{P}(\vec u_{\lambda}(t))$$
for all $t\ge 0$. 
Moreover, by Lemma \ref{lem:ap}, we have 
$$\frac{d}{dt}\mathcal{A}(\vec u_{\lambda}(t))
=\mathcal{P}(\vec u_{\lambda}(t))
\ge \frac{\delta_{\lambda}}{k^*C_2}$$
for all $t\ge 0$, which implies that 
$\mathcal{A}(\vec u_{\lambda}(t))\to \infty$ as $t\to \infty$. 
This contradicts the fact that 
$|\mathcal{A}(\vec u_{\lambda}(t))|\le C_1$ for all $t\ge 0$. 
Hence, $G(\omega t)\vec \phi$ is unstable. 
\end{proof}

\section{Ground States} \label{sect5}

\subsection{Existence and Stability of Ground States}

In this subsection, we briefly recall the existence and stability 
of ground states for \eqref{sp}.  
We define 
\begin{align*}
&\|\vec u\|_{X_{\omega}}^2=\|\nabla \vec u\|_{H}^2+\omega \|\vec u\|_{H}^2, \\
&V(\vec u)=\kappa \|u_1\|_{L^3}^3+\|u_2\|_{L^3}^3
+\frac{3}{2}\gamma \Re \int_{\R^N}u_1^2\overline{u_2}\,dx, \\
&K_{\omega}(\vec u)=\|\vec u\|_{X_{\omega}}^2-V(\vec u)
\end{align*}
for $\vec u\in X$. Then the action $S_{\omega}$ associated with \eqref{sp} is written as 
$$S_{\omega}(\vec u)=\frac{1}{2}\|\vec u\|_{X_{\omega}}^2-\frac{1}{3}V(\vec u).$$
Remark that for $\vec{u}\in X$ satisfying $K_\omega(\vec{u})=0$, one has
\begin{align}\label{qq}
S_\omega(\vec{u})= \frac{1}{6} \| \vec{u}\|_{X_\omega}^2.
\end{align}
Moreover, we define
\begin{align*}
&\mu(\omega)=\inf\{S_{\omega}(\vec u):
\vec u\in X, ~ K_{\omega}(\vec u)=0,~ \vec u\ne (0,0)\}, \\
&\mathcal{M}_{\omega}
=\{\vec \phi\in X:S_{\omega}(\vec \phi)=\mu(\omega),~ K_{\omega}(\vec \phi)=0\}.
\end{align*}

The following Lemma \ref{lem:comp} establishes the existence of a ground state 
for \eqref{sp}. Since it can be proved by the standard variational method 
(see \cite{BrLi,lio1,lio2,wil} and also \cite{CO,pom}), we omit the proof. 

\begin{lemma}\label{lem:comp}
Let $\kappa\in \R$, $\gamma>0$ and $\omega>0$. 
If $\{\vec u_n\}\subset X$ satisfies 
$S_{\omega}(\vec u_n)\to \mu(\omega)$ and $K_{\omega}(\vec u_n)\to 0$, 
then there exist a sequence $\{y_n\}\subset \R^N$ 
and $\vec \phi\in \mathcal{M}_{\omega}$ such that 
$\{\tau_{y_n}\vec u_n\}$ has a subsequence that converges to $\vec \phi$ 
strongly in $X$. Moreover, 
$\mathcal{M}_{\omega}=\mathcal{G}_{\omega}$ and $\mu(\omega)=d(\omega)$. 
As a consequence, the set $\mathcal{G}_\omega$ is not empty.
\end{lemma}

Next, we consider the stability of ground states. 
By the scale invariance of \eqref{sp}, we see that 
$d(\omega)=\omega^{3-N/2}d(1)$ for all $\omega>0$. 
Since $N\le 3$ and $d(1)>0$, we have $d''(\omega)>0$ for all $\omega>0$. 
Using this fact and Lemma \ref{lem:comp}, 
the following Proposition \ref{prop5} can be proved 
by the method of Shatah \cite{sha} (see also \cite{CO}). 
Since it is standard, we omit the proof. 

\begin{proposition}\label{prop5}
Let $\kappa\in \R$ and $\gamma>0$. For any $\omega>0$, 
the set $\mathcal{G}_{\omega}$ is stable in the following sense. 
For any $\varepsilon>0$ there exists $\delta>0$ such that 
if $\vec u_0\in X$ satisfies 
$\mathrm{dist}\,(\vec u_0,\mathcal{G}_{\omega})<\delta$, 
then the solution $\vec u(t)$ of \eqref{eq:co} with $\vec u(0)=\vec u_0$ 
exists for all $t\ge 0$, and satisfies 
$\mathrm{dist}\,(\vec u(t),\mathcal{G}_{\omega})<\varepsilon$ 
for all $t\ge 0$, where we put 
$$\mathrm{dist}\,(\vec u, \mathcal{G}_{\omega})
=\inf\{\|\vec u-\vec \phi\|_X: \vec \phi\in \mathcal{G}_{\omega}\}.$$
\end{proposition}

\subsection{Preliminaries from Elementary Geometry}

In this section, we explain some basic geometric properties 
concerning the line and the ellipse defined by \eqref{eq:ab}.
In the proof of Theorem \ref{thm6}, one has to compare, 
for a given $(\alpha,\beta)\in \mathcal{S}_{\kappa,\gamma}$, 
the quantities $\alpha^2+\beta^2$ and $1$.
This is the purpose of Lemmas \ref{lem31} and \ref{lem32}. 

\begin{lemma}\label{lem34}
Let $\gamma>0$ and $0<r\le 1$, and put 
$$B=\{(x,y)\in ]0,\infty[^2:\gamma x^2+2y^2=2y,~ x^2+y^2=r^2\}.$$
\par \noindent $({\rm 1})$ \hspace{1mm}
If $0<r<1$, then $B$ consists of one point. 
\par \noindent $({\rm 2})$ \hspace{1mm}
If $r=1$ and $0<\gamma<1$, then 
$B=\{(2\sqrt{1-\gamma}/(2-\gamma), \gamma/(2-\gamma))\}$. 
\par \noindent $({\rm 3})$ \hspace{1mm}
If $r=1$ and $\gamma\ge 1$, then $B$ is empty. 
\end{lemma}

\begin{proof}
First we prove (1). Let $0<r<1$. 
Recall that $\gamma x^2+2y^2=2y$ is an ellipse with vertices 
$(x,y)=(0,0)$, $(0,1)$ and $(\pm 1/\sqrt{2\gamma},1/2)$, 
and that $B\subset \{(x,y):0<y<1\}$. 
By the equations in $B$, we have $g(y):=(2-\gamma)y^2-2y+\gamma r^2=0$. 
Since $g(0)=\gamma r^2>0$ and $g(1)=\gamma (r^2-1)<0$, 
the equation $g(y)=0$ has only one solution in $]0,1[$. 
This proves (1). 
Parts (2) and (3) are obtained by direct computations. 
\end{proof}

\begin{lemma}\label{lem31}
Let $(\kappa,\gamma)\in \mathcal{J}_1\cup \mathcal{J}_2$. 
Then, $\alpha_{+}^2+\beta_{-}^2=1$ 
if and only if $(\kappa,\gamma)\in \mathcal{K}_3$. 
\end{lemma}

\begin{proof}
Assume that $(\alpha_{+},\beta_{-})$ satisfies $\alpha_{+}^2+\beta_{-}^2=1$. 
Since $(\alpha_{+},\beta_{-})$ satisfies 
$\gamma \alpha_{+}^2+2\beta_{-}^2=2\beta_{-}$, 
it follows from (2) and (3) of Lemma \ref{lem34} that $0<\gamma<1$ and 
$(\alpha_{+},\beta_{-})=(2\sqrt{1-\gamma}/(2-\gamma), \gamma/(2-\gamma))$. 
Substituting this into $\kappa \alpha_{+}+\gamma \beta_{-}=1$, 
we have $\kappa=\kappa_c(\gamma)$. Thus, $(\kappa,\gamma)\in \mathcal{K}_3$. 
Conversely, it is easy to see that 
$\alpha_{+}^2+\beta_{-}^2=1$ if $(\kappa,\gamma)\in \mathcal{K}_3$. 
\end{proof}

\begin{lemma}\label{lem32}
If $(\kappa,\gamma)\in \mathcal{K}_1$, then $\alpha_{+}^2+\beta_{-}^2<1$. 
\end{lemma}

\begin{proof}
First, we remark that 
the function $f(\kappa,\gamma):=\alpha_{+}^2+\beta_{-}^2-1$ 
is continuous in $\mathcal{J}_1\cup \mathcal{J}_2$, 
and that $\mathcal{K}_1$ is a connected subset 
of $\mathcal{J}_1\cup \mathcal{J}_2$. 
By Lemma \ref{lem31}, $f$ has no zeros in $\mathcal{K}_1$. 
Thus, $f$ has a constant sign in $\mathcal{K}_1$. 
Finally, since $f(0,\gamma)\to -1$ as $\gamma\to \infty$, 
we conclude that $f(\kappa,\gamma)<0$ 
for all $(\kappa,\gamma)\in \mathcal{K}_1$. 
\end{proof}

In the same way as Lemma \ref{lem32}, 
we see that $\alpha_{+}^2+\beta_{-}^2>1$ 
for $(\kappa,\gamma)\in \mathcal{K}_2\cap \mathcal{J}_2$, 
but this fact is not used in what follows. 
The following Lemma \ref{lem33} plays an important role 
in the proof of Lemma \ref{lem35}. 

\begin{lemma}\label{lem33}
Let $(\kappa,\gamma)\in \mathcal{J}_1\cup \mathcal{J}_2$. 
Then, $\gamma \alpha_{+}>\kappa \beta_{-}$.
\end{lemma}

\begin{proof}
Since $\gamma$, $\alpha_{+}$ and $\beta_{-}$ are positive, 
the inequality is trivial for the case $\kappa\le 0$. 
Let $\kappa>0$ and put $D=\kappa^2+2\gamma(\gamma-1)$. 
Then, $D>0$ and 
\begin{align*}
\gamma \alpha_{+}>\kappa \beta_{-} & \iff 
(2-\gamma)\gamma \kappa +\gamma^2\sqrt{D}
>\kappa(\kappa^2+\gamma^2)-\kappa^2\sqrt{D} \\
& \iff (\gamma^2+\kappa^2)\sqrt{D}>\kappa D. 
\end{align*}
Since $(\gamma^2+\kappa^2)^2-\kappa^2D=\gamma^4+2\gamma \kappa^2>0$, 
the last inequality holds. 
\end{proof}

We define 
\begin{equation}\label{def:ell}
\ell=\left\{\begin{array}{ll}
\alpha_{+}^2+\beta_{-}^2 &\quad \text{if} \hspace{2mm} 
(\kappa,\gamma)\in \mathcal{K}_1, \\
1 &\quad \text{if} \hspace{2mm} 
(\kappa,\gamma)\in \mathcal{K}_2\cup \mathcal{K}_3,
\end{array}\right.
\end{equation}
and for a given $(\kappa,\gamma)$,
\begin{align*}
&E_1=\{(x,y)\in ]0,\infty[^2:\kappa x+\gamma y\ge 1\}, \\
&E_2=\{(x,y)\in ]0,\infty[^2:\gamma x^2+2y^2\ge 2y,~ 
x^2+y^2\le \ell\}.
\end{align*}

In Lemmas \ref{lem35} and \ref{lem38}, 
we establish the structure of the set $E_1\cap E_2$ with respect to $(\kappa,\gamma)$.

\begin{lemma}\label{lem35}
If $(\kappa,\gamma)\in \mathcal{K}_1\cup \mathcal{K}_3$, 
then $E_1\cap E_2=\{(\alpha_{+},\beta_{-})\}$. 
\end{lemma}

\begin{proof}
Since it is clear that 
$\{(\alpha_{+},\beta_{-})\}\subset E_1\cap E_2$, 
we prove the inverse inclusion. 
By Lemmas \ref{lem34}, \ref{lem31} and \ref{lem32}, 
we see that the ellipse $\gamma x^2+2y^2=2y$ 
and the circle $x^2+y^2=\alpha_{+}^2+\beta_{-}^2$ 
intersect at only one point $(x,y)=(\alpha_{+},\beta_{-})$ in $\{(x,y):x>0\}$. 
The normal $y=f_1(x)$ and the tangent $y=f_2(x)$ 
of the circle $x^2+y^2=\alpha_{+}^2+\beta_{-}^2$ at the point 
$(x,y)=(\alpha_{+},\beta_{-})$ are given by 
$$f_1(x)=\frac{\beta_{-}}{\alpha_{+}}(x-\alpha_{+})+\beta_{-}, \quad 
f_2(x)=-\frac{\alpha_{+}}{\beta_{-}}(x-\alpha_{+})+\beta_{-},$$
and we see that $E_2\subset 
E_3:=\{(x,y):y\le f_1(x), ~ y\le f_2(x)\}$. 
By Lemma \ref{lem33} and by $\kappa \alpha_{+}+\gamma \beta_{-}=1$, 
we have $-\alpha_{+}/\beta_{-}<-\kappa/\gamma<\beta_{-}/\alpha_{+}$. 
That is, the slope of the line $\kappa x+\gamma y=1$ 
is less than that of the normal $y=f_1(x)$, 
and is greater than or equal to that of the tangent $y=f_2(x)$. 
Recalling that $(\alpha_{+},\beta_{-})$ is on the line $\kappa x+\gamma y=1$, 
we see that $E_1\cap E_2\subset E_1\cap E_3=\{(\alpha_{+},\beta_{-})\}$. 
This completes the proof. 
\end{proof}

\begin{lemma}\label{lem38}
If $(\kappa,\gamma)\in \mathcal{K}_2$, then $E_1\cap E_2$ is empty. 
\end{lemma}

\begin{proof}
First, we consider the case where $\kappa\le 0$ and $0<\gamma\le 1$. 
Then, $E_1\subset \{(x,y):y\ge 1/\gamma\}\subset \{(x,y):y\ge 1\}$, 
and we see that $E_1\cap E_2$ is empty. 

Next, we consider the case where $0<\gamma<1$ and $0<\kappa<\kappa_c(\gamma)$. 
We fix $\gamma\in ]0,1[$ and denote $E_1=E_1(\kappa)$ for $0<\kappa\le \kappa_c$. 
Remark that $E_2$ is independent of $\kappa\in ]0,\kappa_c]$. 
When $\kappa=\kappa_c$, by Lemma \ref{lem35}, 
we have $E_1(\kappa_c)\cap E_2=\{(\alpha_{+},\beta_{-})\}$. 
Moreover, when $0<\kappa<\kappa_c$, 
$E_1(\kappa)$ is strictly smaller than $E_1(\kappa_c)$. 
Thus, we see that $E_1(\kappa)\cap E_2$ is empty if $0<\kappa<\kappa_c$. 
\end{proof}

\subsection{Determination of Ground States}

We are now able to determine the structure of the set $\mathcal{G}_\omega$.
We use an idea of Sirakov \cite{sir} (see also \cite{kik}). 
We denote 
$$\|u\|_{H^1_{\omega}}^2=\|\nabla u\|_{L^2}^2+\omega \|u\|_{L^2}^2,
\quad u\in H^1(\R^N).$$

\begin{lemma}\label{lem41}
Let $\vec u=(u_1,u_2)\in \mathcal{A}_{\omega}$. Then we have 
\begin{align*}
&\|u_1\|_{H^1_{\omega}}^2
=\kappa \|u_1\|_{L^3}^3+\gamma \int_{\R^N}\overline{u_1}^2u_2\,dx, \\
&\|u_2\|_{H^1_{\omega}}^2
=\|u_2\|_{L^3}^3+\frac{\gamma}{2} \int_{\R^N}u_1^2\overline{u_2}\,dx. 
\end{align*}
\end{lemma}

\begin{proof}
The first identity is obtained by 
mutliplying the first equation of \eqref{sp} by $\overline{u_1}$ 
and by integrating by parts. 
The second identity is obtained in the same way. 
\end{proof}

\begin{lemma}\label{lem42}
For any $\omega>0$, $6d(\omega)\le \ell \|\varphi_{\omega}\|_{L^3}^3$, 
where $\ell$ is the number defined by \eqref{def:ell}. 
\end{lemma}

\begin{proof}
Let $(\alpha \varphi_{\omega},\beta \varphi_{\omega})\in \mathcal{A}_{\omega}$. 
Then, we have $K_{\omega}(\alpha \varphi_{\omega},\beta \varphi_{\omega})=0$, 
and so by \eqref{qq}, 
\begin{align*}
S_{\omega}(\alpha \varphi_{\omega},\beta \varphi_{\omega})=\frac{1}{6}
\|(\alpha \varphi_{\omega},\beta \varphi_{\omega})\|_{X_{\omega}}^2
=\frac{\alpha^2+\beta^2}{6}\|\varphi_{\omega}\|_{H^1_{\omega}}^2.
\end{align*}
Moreover, since $\varphi_{\omega}$ is a solution of \eqref{scalar}, 
we have $\|\varphi_{\omega}\|_{H^1_{\omega}}^2=\|\varphi_{\omega}\|_{L^3}^3$, 
and $6S_{\omega}(\alpha \varphi_{\omega},\beta \varphi_{\omega})
=(\alpha^2+\beta^2)\|\varphi_{\omega}\|_{L^3}^3$. 
Finally, by the definitions of $d(\omega)$ and $\ell$, 
we obtain the desired estimate. 
\end{proof}

The following variational characterization of $\varphi_{\omega}$ is well-known 
(see \cite{caz,kik,sir,wil}), and we omit the proof. 

\begin{lemma}\label{lem43} 
Let $\omega>0$. Then 
\begin{equation}\label{min1}
\|\varphi_{\omega}\|_{L^3}
=\inf\left\{{\|v\|_{H^1_{\omega}}^2}/
{\|v\|_{L^3}^2}: v\in H^1(\R^N)\setminus\{0\}\right\}. 
\end{equation}
Moreover, 
\begin{align}
&\{v\in H^1(\R^N):\|v\|_{H^1_{\omega}}^2
=\|v\|_{L^3}^3=\|\varphi_{\omega}\|_{L^3}^3\} \nonumber \\
&=\{e^{i\theta}\varphi_{\omega}(\cdot+y):\theta\in \R,~ y\in \R^N\}.
\label{min2}
\end{align}
\end{lemma}

The next lemma is linked to the key Lemmas \ref{lem35} and \ref{lem38}. 

\begin{lemma}\label{lem44}
Let $(u_1,u_2)\in \mathcal{G}_{\omega}$, and put 
$$a=\|u_1\|_{L^3}/\|\varphi_{\omega}\|_{L^3}, \quad 
b=\|u_2\|_{L^3}/\|\varphi_{\omega}\|_{L^3}.$$
Then, $a\ge 0$, $b>0$, and $(a,b)$ satisfies 
\begin{equation}\label{abab}
a^2\leq a^2(\kappa a+\gamma b), \quad 
2b\leq 2b^2+\gamma a^2, \quad a^2+b^2\le \ell.
\end{equation}
Moreover, 
\par \noindent $({\rm 1})$ \hspace{1mm}
If $(\kappa,\gamma)\in \mathcal{K}_1$, then $(a,b)=(\alpha_{+},\beta_{-})$. 
\par \noindent $({\rm 2})$ \hspace{1mm}
If $(\kappa,\gamma)\in \mathcal{K}_2$, then $(a,b)=(0,1)$. 
\par \noindent $({\rm 3})$ \hspace{1mm}
If $(\kappa,\gamma)\in \mathcal{K}_3$, 
then $(a,b)\in \{(\alpha_{+},\beta_{-}), (0,1)\}$. 
\end{lemma}

\begin{proof}
We first prove \eqref{abab}. 
If $u_2=0$, then the second equation of \eqref{sp} implies $u_1=0$. 
This contradicts $(u_1,u_2)\in \mathcal{A}_{\omega}$. 
Thus, $u_2\ne 0$ and $b>0$. 
By \eqref{min1}, Lemma \ref{lem41} and the H\"older inequality, 
we have 
\begin{align}
\|\varphi_{\omega}\|_{L^3}\|u_1\|_{L^3}^2
&\leq \|u_1\|_{H^1_{\omega}}^2
=\kappa \|u_1\|_{L^3}^3 +\gamma \int_{\R^N}\overline{u_1}^2u_2\,dx 
\nonumber \\
&\leq \kappa \|u_1\|_{L^3}^3 +\gamma \|u_1\|_{L^3}^2 \|u_2\|_{L^3}, 
\label{vin1}
\end{align}
which provides $a^2\leq a^2(\kappa a+\gamma b)$. 
In the same way, we have 
\begin{align}
\|\varphi_{\omega}\|_{L^3} \|u_2\|_{L^3}^2
&\leq \|u_2\|_{H^1_{\omega}}^2
=\|u_2\|_{L^3}^3 +\frac{\gamma}{2} \int_{\R^N}u_1^2\overline{u_2}\,dx 
\nonumber \\
&\leq \|u_2\|_{L^3}^3+\frac{\gamma}{2} \|u_1\|_{L^3}^2 \|u_2\|_{L^3}.
\label{vin2}
\end{align}
Since $b>0$, this gives $2b\leq 2b^2+\gamma a^2$. 
Finally, by Lemma \ref{lem42} and \eqref{min1}, we obtain 
\begin{equation}\label{vin3}
\ell \|\varphi_{\omega}\|_{L^3}^3
\ge 6d(\omega)=6S_{\omega}(\vec u)=\|\vec u\|_{X_{\omega}}^2
\geq \|\varphi_{\omega}\|_{L^3}\Big( \|u_1\|_{L^3}^2+\|u_2\|_{L^3}^2\Big),
\end{equation}
which implies $a^2+b^2\le \ell$. Hence, \eqref{abab} is proved. 

We now prove (1), (2) and (3). 
Let $(\kappa,\gamma)\in \mathcal{K}_1$. 
Then, since $\ell<1$, by Lemma \ref{lem44}, 
we see that $a>0$ and $(a,b)\in E_1\cap E_2$. 
Thus, (1) follows from \eqref{abab}. 
Next, let $(\kappa,\gamma)\in \mathcal{K}_2$. 
Suppose that $a>0$. Then, by \eqref{abab}, 
we have $(a,b)\in E_1\cap E_2$. 
However, this contradicts Lemma \ref{lem38}.
Thus, we have $a=0$ and $b=1$, which proves (2). 
Part (3) can be proved similarly. 
\end{proof}

\begin{proof}[Proof of Theorem \ref{thm6}]
We consider the case $(\kappa,\gamma)\in \mathcal{K}_1$. 
Let $\vec u\in \mathcal{G}_{\omega}$. 
By \eqref{vin1}, \eqref{vin2} and Lemma \ref{lem44}, we see that 
\begin{align}
&\|u_1/\alpha_{+}\|_{H^1_{\omega}}^2
=\|\varphi_{\omega}\|_{L^3}^3=\|u_1/\alpha_{+}\|_{L^3}^3, \label{win1} \\
&\|u_2/\beta_{-}\|_{H^1_{\omega}}^2
=\|\varphi_{\omega}\|_{L^3}^3=\|u_2/\beta_{-}\|_{L^3}^3, \label{win2} \\
&\int_{\R^N}u_1^2\overline{u_2}\,dx
=\|u_1\|_{L^3}^2\|u_2\|_{L^3}. \label{win3}
\end{align}
By \eqref{min2} and by \eqref{win1} and \eqref{win2}, 
there exist $(\theta_1,y_1)$, 
$(\theta_2,y_2)\in \R\times \R^N$ such that 
$u_1=e^{i\theta_1}\alpha_{+}\varphi_{\omega}(\cdot+y_1)$ and 
$u_2=e^{i\theta_2}\beta_{-}\varphi_{\omega}(\cdot+y_2)$. 
Moreover, by \eqref{win3}, we see that 
$2\theta_1-\theta_2\in 2\pi \Z$ and $y_1=y_2$. 
Thus, we have $\mathcal{G}_{\omega}\subset \mathcal{G}_{\omega}^1$. 
Since $\mathcal{G}_{\omega}$ is not empty, (1) is proved. 
(2) and (3) can be proved in the same way. 
\end{proof}

Finally, Theorem \ref{thm5} is obtained as a corollary of Theorem \ref{thm6}. 

\begin{proof}[Proof of Theorem \ref{thm5}]
Let $\kappa\le 0$ and $\gamma=1$. 
Then, $(\kappa,\gamma)\in \mathcal{K}_2$, 
and by Theorem \ref{thm6}, $\mathcal{G}_{\omega}=\mathcal{G}_{\omega}^0$. 
Thus, Theorem \ref{thm5} follows from Proposition \ref{prop5}. 
\end{proof}

\vspace{3mm}\noindent \textbf{Acknowledgments.}
The work of the second author was supported 
by JSPS Excellent Young Researchers Overseas Visit Program 
and by JSPS KAKENHI (21540163).


\begin{thebibliography}{99}

\bibitem{BC}H. Berestycki and T. Cazenave, 
{\em Instabilit\'{e} des \'{e}tats stationnaires dans les \'{e}quations de
Schr\"odinger et de Klein-Gordon non lin\'{e}aires}, 
{C. R. Acad. Sci. Paris s\'{e}r. I Math.} {\bf 293} (1981) 489--492. 

\bibitem{BrLi}H. Brezis and E. H. Lieb, 
{\em Minimum action solutions of some vector field equations}, 
{Comm. Math. Phys.} {\bf 96} (1984) 97--113. 

\bibitem{caz}T. Cazenave,
{\em Semilinear Schr\"odinger equations}, 
{Courant Lecture Notes in Mathematics 10}, 
Amer. Math. Soc., 2003. 

\bibitem{CL}T. Cazenave and P. L. Lions,
{\em Orbital stability of standing waves 
for some nonlinear Schr\"odinger equations,}
{Comm. Math. Phys.} {\bf 85} (1982) 549--561.

\bibitem{CC1} M. Colin and T. Colin, 
{\em On a quasi-linear Zakharov system describing laser-plasma interactions}, 
{Differential Integral Equations} {\bf 17} (2004) 297--330.

\bibitem{CC2} M. Colin and T. Colin, 
{\em A numerical model for the Raman amplification for laser-plasma interaction}, 
{J. Comput. App. Math.} {\bf 193} (2006) 535--562.

\bibitem{CCO1}M. Colin, T. Colin and M. Ohta, 
{\em Stability of solitary waves for a system of 
nonlinear Schr\"odinger equations with three wave interaction}, 
{Ann. Inst. H. Poincar\'{e}, Anal. Non Lin\'{e}aire} 
{\bf 26} (2009) 2211--2226. 

\bibitem{CCO2}M. Colin, T. Colin and M. Ohta, 
{\em Instability of standing waves for a system of 
nonlinear Schr\"odinger equations with three-wave interaction}, 
{Funkcial. Ekvac.} {\bf 52} (2009) 371--380. 

\bibitem{CO}M. Colin and M. Ohta, 
{\em Stability of solitary waves for derivative nonlinear Schr\"odinger equation},
{Ann. Inst. H. Poincar\'{e}, Anal. Non Lin\'{e}aire} {\bf 23} (2006) 753--764. 

\bibitem{CR}M. G. Crandall and P. H. Rabinowitz, 
{\em Bifurcation from simple eigenvalues}, 
{J. Funct. Anal.} {\bf 8} (1971) 321--340. 

\bibitem{GSS1}M. Grillakis, J. Shatah and W. Strauss, 
{\em Stability theory of solitary waves in the presence of symmetry, I}, 
{J. Funct. Anal.} {\bf 74} (1987) 160--197. 

\bibitem{kik}H. Kikuchi,
{\em Orbital stability of semitrivial standing waves 
for the Klein-Gordon-Schr\"{o}dinger system}, preprint. 

\bibitem{lio1}P. L. Lions,
{\em The concentration-compactness principle in the calculus of variations,
the locally compact case, part I},
{Ann. Inst. H. Poincar\'{e}, Anal. Nonlin\'{e}aire} {\bf 1} (1984) 109--145.

\bibitem{lio2}P. L. Lions,
{\em The concentration-compactness principle in the calculus of variations,
the locally compact case, part II}, 
{Ann. Inst. H. Poincar\'{e}, Anal. Nonlin\'{e}aire} {\bf 1} (1984) 223--282.

\bibitem{mae}M. Maeda, 
{\em Instability of bound states of nonlinear Schr\"{o}dinger equations 
with Morse index equal to two}, 
{Nonlinear Anal.} {\bf 72} (2010) 2100--2113. 

\bibitem{oht}M. Ohta, 
{\em Instability of bound states for abstract 
nonlinear Schr\"{o}dinger equations}, 
preprint, arXiv:1010.1511. 

\bibitem{pom}A. Pomponio, 
{\em Ground states for a system of nonlinear Schr\"odinger equations 
with three wave interaction}, 
{J. Math. Phys.} {\bf 51} (2010) 093513, 20pp. 

\bibitem{sha}J. Shatah, 
{\em Stable standing waves of nonlinear Klein-Gordon equations}, 
{Comm. Math. Phys.} {\bf 91} (1983) 313--327. 

\bibitem{SS}J. Shatah and W. Strauss,
{\em Instability of nonlinear bound states},
{Comm. Math. Phys.} {\bf 100} (1985) 173--190.

\bibitem{sir}B. Sirakov, 
{\em Least energy solitary waves for a system of nonlinear Schr\"odinger 
equations in $\R^N$}, 
{Comm. Math. Phys.} {\bf 271} (2007) 199--221. 

\bibitem{yew}A. C. Yew, 
{\em Stability analysis of multipulses in nonlinearly-coupled 
Schr\"{o}dinger equations}, 
{Indiana Univ. Math. J.} {\bf 49} (2000) 1079--1124. 

\bibitem{wei1}M. I. Weinstein, 
{\em Modulational stability of ground states 
of nonlinear Schr\"odinger equations}, 
{SIAM J. Math. Anal.} {\bf 16} (1985) 472--491. 

\bibitem{wei2}M. I. Weinstein, 
{\em Lyapunov stability of ground states 
of nonlinear dispersive evolution equations}, 
{Comm. Pure Appl. Math.} {\bf 39} (1986) 51--68. 

\bibitem{wil}M. Willem, 
{\em Minimax theorems}, 
Progress in Nonlinear Differential Equations and their Applications, 24. 
Birkh\"auser, Boston, MA, 1996. 

\end{thebibliography}
\end{document}